\documentclass{amsart}
\usepackage{amsmath,amssymb,amsfonts,graphicx,hyperref,color}

\title{Inhomogeneous extreme forms}

\author{Mathieu Dutour Sikiri\'c}
\address{M.~Dutour Sikiri\'c, Rudjer Boskovi\'c Institute, Bijenicka 54, 10000 Zagreb, Croatia}
\email{mdsikir@irb.hr}

\author{Achill Sch\"urmann}
\address{A.~Sch\"urmann, 
Universit\"at Rostock, Institut f\"ur Mathematik, 18051 Rostock, Germany}
\email{achill.schuermann@uni-rostock.de}

\author{Frank Vallentin} 
\address{F.~Vallentin, Delft Institute of Applied Mathematics, Technical University of Delft, P.O. Box 5031, 2600 GA Delft, The Netherlands}
\email{f.vallentin@tudelft.nl}

\thanks{The work of the first author has been supported by the
  Croatian Ministry of Science, Education and Sport under contract
  098-0982705-2707. The second and the third author were supported by
  the Deutsche Forschungsgemeinschaft (DFG) under grant SCHU 1503/4-2.
The third author was supported by Vidi grant 639.032.917 from the
Dutch Organization for Scientific Research (NWO)}

\subjclass{11H55, 52C17} 

\keywords{lattices, Delone polytopes, spherical $t$-designs, sphere packing, sphere covering, Voronoi reduction theory}

\date{September 24, 2011}

\input cyracc.def
\font\tencyr=wncyr10
\def\cyr{\tencyr\cyracc}

\newcommand{\R}{\ensuremath{\mathbb{R}}}
\newcommand{\Z}{\ensuremath{\mathbb{Z}}}
\newcommand{\EC}{\ensuremath{\mathcal{E}}}
\newcommand{\AGL}{\ensuremath{\mathsf{AGL}}}
\newcommand{\GL}{\ensuremath{\mathsf{GL}}}

\newtheorem{defin}{Definition}[section]
\newtheorem{definition}[defin]{Definition}
\newtheorem{proposition}[defin]{Proposition}
\newtheorem{theorem}[defin]{Theorem}

\newtheorem{lemma}[defin]{Lemma}

\DeclareMathOperator{\conv}{conv}

\DeclareMathOperator{\Min}{Min}

\DeclareMathOperator{\Stab}{Stab}
\DeclareMathOperator{\trace}{trace}
\DeclareMathOperator{\Del}{Del}
\DeclareMathOperator{\vertex}{vert}
\DeclareMathOperator{\grad}{grad}
\DeclareMathOperator{\interior}{int}
\DeclareMathOperator{\cone}{cone}

\DeclareMathOperator{\vol}{vol}
\DeclareMathOperator{\lin}{lin}
\DeclareMathOperator{\Aut}{Aut}

\DeclareMathOperator{\ev}{ev}
\DeclareMathOperator{\aff}{aff}
\DeclareMathOperator{\dist}{dist}
\DeclareMathOperator{\diag}{diag}

\newcommand{\cv}{{\mathcal V}}

\DeclareMathOperator{\tdim}{dim}

\begin{document}

\begin{abstract}
G.F.~Voronoi (1868--1908) wrote two memoirs in which he describes two reduction theories for lattices, well-suited for sphere packing and covering problems. In his first memoir a characterization of locally most economic packings is given, but a corresponding result for coverings has been missing. In this paper we bridge the two classical memoirs.

By looking at the covering problem from a different perspective, we discover the missing analogue. Instead of trying to find lattices giving economical coverings we consider lattices giving, at least locally, very uneconomical ones. We classify local covering maxima up to dimension~$6$ and prove their existence in all dimensions beyond.

New phenomena arise: Many highly symmetric lattices turn out to give uneconomical coverings; the covering density function is not a topological Morse function. Both phenomena are in sharp contrast to the packing problem.  
\end{abstract}

\maketitle

\section{Introduction}

A basis of the $n$-dimensional Euclidean space $\R^n$ defines a lattice consisting of all integer linear combinations. A lattice defines a sphere packing in the following way: One centers congruent balls at the lattice points with maximum radius such that interiors do not intersect. Similarly, it defines a sphere covering: One places congruent balls with minimum radius such that each point in $\R^n$ is covered by a ball.

The \emph{(lattice sphere) packing problem} asks for a lattice which gives the most economical packing, i.e.\ one which maximizes the fraction of space covered by the balls.  The \emph{(lattice sphere) covering problem} asks for a lattice which gives the most economical covering, i.e.\ one which minimizes the average number of balls covering a point in~$\R^n$.

Many researchers were attracted by the packing problem. One important reason for this is that low-dimensional lattices which give good packings are often related to objects of exceptional beauty in combinatorics, geometry, and number theory. A vivid account of this is the monograph \cite{CS} by Conway and Sloane with over 100 pages of references which since the appearance of its first edition in 1988 spurred a tremendous amount of activity.

Our computational studies in \cite{Vallentin}, \cite{SV1}, \cite{SV2}, \cite{DSV} show that the covering problem behaves very differently. Many of the best known coverings could only be discovered with computer assistance. They were found by a numerical convex continuous optimization procedure; some of them do not have a rational representation, and their beauty is not immediately apparent.

Furthermore, in \cite{SV1} it came as a surprise that the root lattice $\mathsf{E}_8$ does not even give a locally optimal covering whereas the Leech lattice $\Lambda_{24}$ does. Both lattices are the unique optimum, up to scaling and isometries, for the lattice packing problem which was proved by Blichfeldt \cite{Blichfeldt} (optimality of $\mathsf{E}_8$), Vetchinkin \cite{Vetchinkin}  (uniqueness of $\mathsf{E}_8$) and Cohn, Kumar \cite{CohnKumar} (optimality and uniqueness of $\Lambda_{24}$). In many respects both lattices behave similarly.  The shortest vectors of both lattices give spherical point configurations which are optimal for many other extremal questions in geometry, like the kissing number problem and more generally for potential energy minimization which is proved in Cohn and Kumar's work on universally optimal point configurations on spheres \cite{CohnKumar2}.

From further experimental studies we saw that $\mathsf{E}_8$ is almost a local covering maximum, that is, the covering density decreases for almost all perturbations of $\mathsf{E}_8$. We say that $\mathsf{E}_8$ is a covering pessimum. This raised the question: Do local covering maxima exist (although local packing minima do not exist)? The first local covering maximum $\mathsf{E}_6$ is found in~\cite{Sch}.

In this paper we develop the theory of local covering maxima. It turns out that our theory gives a new link between Voronoi's two classical memoirs \cite{Vor1}, \cite{Vor2}. 

We think that this new theory of local covering maxima is interesting for several reasons: First of all it shows what happens to the ``nice'' lattices, like $\mathsf{D}_4$, $\mathsf{E}_6$, $\mathsf{E}_7$, $\mathsf{E}_8$, $\mathsf{K}_{12}$, $\mathsf{BW}_{16}$, $\Lambda_{24}$, in the theory of lattice coverings: With the exception of the Leech lattice, all these ``nice'' lattices give locally very uneconomical sphere coverings. Lattices which have large covering density also come up in connection to Minkowski's conjecture. It states that every lattice $L \subseteq \R^n$ with $\det L = 1$ satisfies
\begin{equation*}
\sup_{x \in \R^n} \inf_{y \in L} \left|(x_1 - y_1) \cdots (x_n - y_n)\right| \leq 2^{-n},
\end{equation*}
and equality holds only for $L = \diag(a_1, \ldots, a_n) \Z^n$ with $|a_1 \cdots a_n| = 1$. Curtis T.~McMullen \cite{McMullen} showed that Minkowski's conjecture follows from the following covering conjecture: The (normalized) covering density of every $n$-dimensional lattice which is generated by its minimal vectors is bounded above by $\sqrt{n}/2$ and equality holds only for lattices which are similar to the standard lattice $\Z^n$. Based on the notions developed in this paper, the second author describes an algorithm to decide the covering conjecture for every fixed dimension $n$ in \cite[Chapter 5.7]{Sch}.

In Section~\ref{sec:extremality} we start by formulating a characterization of local covering maxima in the spirit of Voronoi. In \cite{Vor1} Voronoi gives a similar characterization of local packing maxima extending earlier work of Korkine and Zolotarev. Then, Section~\ref{sec:proof} contains a proof of our characterization. It is based on using the Karush-Kuhn-Tucker condition from nonlinear optimization.

In Section~\ref{sec:examples} we formulate and prove a sufficient condition for being a local covering maximum in the spirit of Venkov's theory of strongly perfect lattices: It uses the $t$-design property of spherical point configurations. In \cite{Venkov} Venkov gives a similar condition for local packing maxima. It turns out that many interesting lattices satisfy this condition.

In Section~\ref{sec:finiteness} we show that there are only finitely many local covering maxima in every dimension and we give a classification which is complete up to dimension~$6$. For dimension $7$ and $8$ we give a list of all known local covering maxima. There is strong numerical evidence that these lists are complete.

One important difference between the packing problem and the covering problem is discussed in Section~\ref{sec:pessima}: Ash \cite{Ash} proved that the packing density function is a topological Morse function. We show that the covering density function does not have this property if the dimension is at least four.

In the last section we give and analyze a construction showing that there are local covering maxima in all dimensions~$n \geq 6$.

\section{Extremality = Perfectness and Eutaxy}
\label{sec:extremality}

In his first memoir Voronoi gives a characterization of locally optimal packings, building on previous works by Korkine and Zolotarev. For this he uses the notions of extremality, perfectness and eutaxy, which are naturally defined in the language of positive definite quadratic forms (PQFs).

Some preliminaries: There is a one-to-one correspondence between lattice bases up to orthogonal transformations and PQFs by taking the Gram matrix of the lattice basis. We identity the space of quadratic forms in $n$ variables with the space of real symmetric $n \times n$-matrices. It is an $\binom{n+1}{2}$-dimensional Euclidean space with inner product $\langle Q, Q' \rangle = \trace(QQ')$, where $Q$ and $Q'$ are quadratic forms. By this identification we can evaluate a quadratic form $Q$ at a vector $x \in \R^n$ by
\begin{equation*}
Q[x] = x^t Q x = \langle Q, xx^t \rangle.
\end{equation*}

Now we review Voronoi's characterization for the homogeneous packing case  where we refer to the monographs \cite{Mar} of Martinet and \cite{Sch} of Sch\"urmann for proofs and further information. Then we present our characterization for the inhomogeneous covering case.

\subsection{Homogeneous case}

Let $Q$ be a positive definite quadratic form in $n$ variables. The
\emph{Hermite invariant} of $Q$ is
\begin{equation*}
\gamma(Q) = \frac{\lambda(Q)}{(\det Q)^{1/n}},
\end{equation*}
where
\begin{equation*}
\lambda(Q) = \min_{v \in \Z^n \setminus \{0\}} Q[v] ,
\end{equation*}
is the \emph{homogeneous minimum} of $Q$. It is
scale-invariant. Maximizing the packing density among lattices is
equivalent to maximizing the Hermite invariant among PQFs.

Voronoi gave a characterization of the local maxima of the Hermite
invariant using the geometry of the \emph{shortest vectors}
\begin{equation*}
\Min Q = \{v \in \Z^n : Q[v] = \lambda(Q)\}.
\end{equation*}

\begin{definition}
Let $Q$ be a PQF.
\begin{itemize}
\item[(i)] It is called \emph{extreme} if it is a local maximum of the
  Hermite invariant.
\item[(ii)] It is called \emph{perfect} if the linear space
  spanned by
\begin{equation*}
\{vv^t : v \in \Min Q\}
\end{equation*}
has maximal possible rank $\binom{n+1}{2}$.
\item[(iii)] It is called \emph{eutactic} if there are positive
  constants $\alpha_v$ so that
\begin{equation*}
  Q^{-1} = \sum_{v \in \Min Q} \alpha_v vv^t.
\end{equation*}
It is called \emph{semieutactic} if the constants are nonnegative, and \emph{weakly eutactic} if the constants are real, i.e.\ if they exist at all.
\end{itemize}
\end{definition}

The extended notion of semieutaxy and weak eutaxy is due to Berg\'e and Martinet \cite{BergeMartinet}

\begin{theorem}[Voronoi \cite{Vor1}]
\label{th:Voronoi}
A PQF is extreme if and only if it is perfect and eutactic.
\end{theorem}

\subsection{Inhomogeneous case}

We define the \emph{inhomogeneous Hermite invariant} of a PQF $Q$ as
\begin{equation*}
\gamma_i(Q) = \frac{\mu(Q)}{(\det Q)^{1/n}},
\end{equation*}
where
\begin{equation*}
\mu(Q) = \max_{x \in\R^n} \min_{v\in\Z^n} Q[x-v]
\end{equation*}
is the \emph{inhomogeneous minimum} of $Q$. Like $\gamma$ it is
scale-invariant. Finding extrema for the covering density among
lattices is equivalent to finding extrema for the inhomogeneous
Hermite invariant among PQFs.

In the literature, so far only the local minima of the inhomogeneous
Hermite invariant have been considered, as they give economical
coverings. However, to link the homogeneous with the inhomogeneous
case we have to consider the local maxima.

In this paper we characterize local maxima of the inhomogeneous
Hermite invariant using the geometry of closest vectors. For each point $c \in \R^n$
attaining $\mu(Q)$ we define the \emph{closest vectors}
\begin{equation*}
\Min_c Q = \{v \in \Z^n : Q[v-c] = \mu(Q)\}.
\end{equation*}
Geometrically, the closest vectors give the vertices of the \emph{Delone (Cyrillic: {\cyr Delone}, French: Delaunay) polytope} defined by the PQF~$Q$ which has center~$c$: We have $Q[v-c] = \mu(Q)$ only for $v \in \Min_c Q$ and for all other lattice points $v \in \Z^n$ we have strict inequality $Q[v - c] > \mu(Q)$. The set of all Delone polytopes is called the \textit{Delone subdivision} of~$Q$ which is a $\Z^n$-periodic polyhedral subdivision of $\R^n$. The inhomogeneous minimum of $Q$ is at the same time the maximum squared circumradius of its Delone polytopes. 

\begin{definition}
Let $Q$ be a PQF.
\begin{itemize}
\item[(i)] It is called \emph{inhomogeneous extreme} if it is a local
  maximum of the inhomogeneous Hermite invariant.

\item[(ii)] It is called \emph{inhomogeneous perfect}, if for
  each $c \in \R^n$ attaining $\mu(Q)$, the linear space spanned by
\begin{equation*}
  \left\{\begin{pmatrix}1\\v\end{pmatrix}\begin{pmatrix}1\\v\\\end{pmatrix}^t
  : v \in \Min_c Q\right\}
\end{equation*}
 has maximal possible rank $\binom{n+2}{2}-1$.
\item[(iii)] It is called \emph{inhomogeneous eutactic}, if for
  each $c \in \R^n$ attaining $\mu(Q)$, there are positive constants
  $\alpha_v$ so that
\begin{equation*}
\begin{pmatrix}
1 & c^t\\
c & cc^t + \frac{\mu(Q)}{n} Q^{-1}
\end{pmatrix}
 = \sum_{v \in \Min_c Q} \alpha_v \begin{pmatrix}1\\v\end{pmatrix}\begin{pmatrix}1\\v\end{pmatrix}^t.
\end{equation*}
It is called \emph{inhomogeneous semieutactic} if the constants are nonnegative, and \emph{inhomogeneous weakly eutactic} if the constants are real, i.e.\ if they exist.
\end{itemize}
\end{definition}

Now we are ready to state our principal result.

\begin{theorem}
  \label{th:principal}
  A PQF is inhomogeneous extreme if and only if it is inhomogeneous
  perfect and inhomogeneous eutactic.
\end{theorem}

We prove this theorem in Section~\ref{sec:proof} after giving a reformulation in the following subsection.

Let us contrast this characterization to the known characterization of PQFs which give local \emph{minima}. Barnes and Dickson \cite{BarnesDickson} gave such a characterization of PQFs in the case of generic PQF $Q$, i.e.\ if all Delone polytopes of $Q$ are simplices: 

\emph{A generic PQF $Q$ is local minimum for $H$ if and only if one can write
\begin{equation*}
Q^{-1} = \sum_c \lambda_c \left(\sum_{i=0}^n \alpha_i v_iv_i^t - cc^t \right),
\end{equation*}
with nonnegative $\lambda_c$  where the sum goes over all $c$ attaining $\mu(Q)$ and where $\Min_c(Q) = \conv\{v_0, \ldots, v_n\}$ and where $\alpha_i$ are so that $\sum_{i=0}^n \alpha_i = 1$ and $c = \sum_{i=0}^n \alpha_i v_i$.}

Hence, this characterization resembles (semi-)eutaxy; there is no perfectness here. This and the other non-generic cases where the Delone polytopes are not all simplices are discussed in \cite[Chapter 5.2.4]{Sch}.

\subsection{Quadratic functions}

Before we go on, a remark why in the definition of inhomogeneous
perfect forms the maximal possible rank is $\binom{n+2}{2}-1$ instead of $\binom{n+2}{2}$
is in order: It is $\binom{n+2}{2}-1$ because the vectors $v$ of $\Min_c Q$
satisfy the equation $Q[v-c] = \mu(Q)$ which translates into one
\emph{linear} equation in the space of quadratic functions. This
observation, due to Erdahl and Ryshkov~\cite{Erd}, \cite{ER},
\cite{RE}, will be the key to the proof of our principal result. Let
us elaborate on this.

Instead of using one quadratic form, which (implicitly) defines the
inhomogeneous minimum $\mu(Q)$ and the points $c \in \R^n$ attaining
$\mu(Q$), we make things explicit by using several quadratic
functions; one for each $c$. We shall explain the exact relation between a PQF and ``its'' quadratic functions in Section~\ref{ssec:relation} once we have all necessary definitions.

A \emph{quadratic function} in $n$ variables can be
written as
\begin{equation*}
  f(x) = \alpha_f + 2b_f \cdot x + Q_f[x],
\end{equation*}
where $\alpha_f \in \R$, $b_f \in \R^n$, and $Q_f$ is a quadratic form
in $n$ variables. By $b_f \cdot x$ we denote the standard inner product of the two $n$-dimensional vectors~$b_f$ and~$x$. We equip the space of quadratic functions with the
inner product
\begin{equation*}
  (f,g) = \alpha_f \alpha_g + 2 b_f \cdot b_g + \langle Q_f, Q_g \rangle.
\end{equation*}
For $x \in \R^n$ we define the quadratic function
\begin{equation*}
  \ev_x(y) = (1 + x \cdot y)^2,
\end{equation*}
which can be used to evaluate a quadratic function $f$ at $x$ by
$(\ev_x, f) = f(x)$. We define the \emph{Erdahl cone} by
\begin{equation*}
\EC_{\geq 0} = \{f : \text{$f(v) \geq 0$ for all $v \in \Z^n$}\}.
\end{equation*}
If a quadratic function $f$ lies in the Erdahl cone, then $Q_f$ is
positive semidefinite (see e.g.\ \cite[Proposition 1.3]{Erd}). We
define the \emph{positive Erdahl cone} by
\begin{equation*}
\EC_{>0} = \{f \in \EC_{\geq 0}: \text{$Q_f$ is positive definite}\}.
\end{equation*}

Let $f$ be a quadratic function lying in the Erdahl cone.  The zero
set of $f$ is an ellipsoid whose interior is free of integral points, points lying in~$\Z^n$. The convex hull of the integral zeroes of $f$ is called the
\emph{Delone polyhedron of $f$},
\begin{equation*}
\Del f = \conv\{v \in \Z^n : f(v) = 0\}.
\end{equation*}
Note that a Delone polyhedron might be empty, bounded or unbounded. We
define the function
\begin{equation*}
\mu(f) = -\min_{x \in \R^n} f(x) = \max_{x \in \R^n} -f(x).
\end{equation*}
We will make extensive use of the fact that $\mu$ is a convex function. This follows because evaluation is linear in~$f$. The function $\mu$ is negative exactly for those $f$ having an empty zero set so that the Delone polyhedron of~$f$ is empty.

Let $f$ be a quadratic function lying in the positive Erdahl cone. If
the zero set of $f$ is a non-degenerate ellipsoid (i.e.\ it is
non-empty and bounded), then its center is $c_f = -Q_f^{-1} b_f$ and
its squared circumradius (with respect to $Q_f$) is $\mu(f)$.  In this case
one can write
\begin{equation*}
  f(x) = Q_f[x - c_f] - \mu(f), \quad \text{and} \quad \mu(f) = Q_f[c_f] - \alpha_f.
\end{equation*}
The \emph{Hermite invariant} of $f \in \EC_{>0}$  is
\begin{equation*}
H(f) = \frac{\mu(f)}{(\det Q_f)^{1/n}}.
\end{equation*}
Note that it is invariant under multiplication by positive scalars.

\begin{definition}
\label{def:perfect_eutactic}
  Let $f$ be a quadratic function lying in the positive Erdahl cone.
\begin{itemize}
\item[(i)] It is called \emph{extreme} if it is a
  local maximum of the Hermite invariant.
\item[(ii)] It is called \emph{perfect}, if the linear space spanned
  by $\ev_v$, with $v \in \vertex \Del f$, has maximal possible rank
  $\binom{n+2}{2}-1$.
\item[(iii)] It is called \emph{eutactic} if there
  are positive real numbers $\alpha_v$, with $v \in \vertex \Del f$,
  so that the following conditions hold
\begin{equation*}
\sum_{v \in \vertex \Del f} \alpha_v \ev_v = \ev_{c_{f}} + \frac{\mu(f)}{n}Q_f^{-1}.
\end{equation*}
It is called \emph{semieutactic} if the constants are nonnegative, and \emph{weakly eutactic} if the constants are real, i.e.\ if they exist.
\end{itemize}
\end{definition}

The equation in the definition of eutaxy (iii) has the following geometric interpretation: A negative multiple of the gradient of the function $H$, which is given on the right hand side (see Lemma~\ref{lem:gradient}), lies in the interior of the \emph{inhomogeneous Voronoi cone}
\begin{equation*}
\cv(f) = \cone\{\ev_x : f(x) = 0, x \in \Z^n\}.
\end{equation*}

\subsection{Relation between quadratic forms and functions}
\label{ssec:relation}

Let $Q$ be a PQF and $c \in \R^n$ be a point
attaining the inhomogeneous minimum $\mu(Q)$. Then the closest vectors
$\Min_c Q$ are the vertices of the Delone polytope $\Del f$ of the
quadratic function $f$ given by $Q_f = Q$, $b_f = Q^{-1}c$, $\mu(f)
= \mu(Q)$. Hence, the inhomogeneous minimum of $Q$ is
\begin{equation*}
  \mu(Q) = \max\{\mu(f) : \text{$f$ quadratic function with $Q_f = Q$}\}.
\end{equation*}
A side remark: The convexity of $f \mapsto \mu(f)$ immediately implies the convexity of $Q \mapsto \mu(Q)$, i.e.\ the main result of Delone, Dolbilin, Ryshkov, Shtogrin in \cite{DDRS}, see also \cite[Proposition 7.1]{SV2} or \cite[Proposition 5.1]{Sch}.

We can reformulate the definition of inhomogeneous perfectness and
eutaxy: A PQF $Q$ is inhomogeneous perfect if all quadratic functions
$f$ with $Q_f = Q$ and $\mu(f) = \mu(Q)$ are perfect. A PQF $Q$ is
inhomogeneous eutactic if all quadratic functions $f$ with $Q_f = Q$
and $\mu(f) = \mu(Q)$ are eutactic. With this,
Theorem~\ref{th:principal} follows immediately from the following
theorem.

\begin{theorem}
\label{th:qf_principal}
  A quadratic function lying in the positive Erdahl cone is extreme if and only if it is perfect and
  eutactic.
\end{theorem}

\subsection{Relation to lattices}

It is well-known that there is a one-to-one correspondence between notions for PQFs (up to unimodular transformations) and notions of lattices (up to orthogonal transformations) which we briefly summarize in the following table:

\medskip

\begin{tabular}{l|l}
{\bf PQF} & {\bf lattice}\\
\hline
determinant & volume of fundamental domain\\
homogeneous minimum & packing radius\\
Hermite invariant & packing density\\
inhomogeneous minimum & covering radius\\
inhomogenous Hermite invariant & covering density
\end{tabular}

\medskip

The relation between quadratic functions and lattices is not that close. Although we use quadratic functions to describe individual Delone polytopes (and so individual vertices of the Voronoi cell of a lattice), some quadratic functions correspond to Delone polytopes, others do not.

\section{Proof of Theorem~\ref{th:qf_principal}}
\label{sec:proof}

The proof of our principal theorem is an analysis of local maxima of a
differentiable function satisfying inequality constraints. We first
recall some background from nonlinear optimization: sufficient and
necessary criteria for a function to have a local maximum. Then we
specialize this to our situation of the Hermite invariant of a quadratic function.

\subsection{Nonlinear optimization}

We just state the result and refer to any book on nonlinear
optimization for more details, e.g.\ the book by Boyd and Vandenberghe \cite[Chapter 5]{BV}.  

Let $E$ be a Euclidean space with inner product $x \cdot y$ and let $p: E \to \R$ and $q_1, \ldots, q_k : E \to \R$ be differentiable functions.  Assume, we want to determine whether or not $p$ has a local maximum $x_0$ on the boundary of the set \begin{equation*}
  G = \{ x \in E : q_i(x) \geq 0 \mbox{ for } i=1,\ldots,k \}.
\end{equation*}
In a sufficiently small neighborhood of $x_0$, the functions $p$ and
$q_i$ can be linearized and approximated by affine functions:
\begin{equation*}
x \mapsto p(x_0) + (\grad p)(x_0) \cdot (x-x_0).
\end{equation*}
We define the \emph{normal cone} of $G$ at $x_0$ by
\begin{equation*}
N(x_0)= \cone \{-(\grad q_i)(x_0) : i=1,\ldots,k \}.
\end{equation*}

\begin{proposition}
\label{prop:nlp}
  Suppose $x_0$ satisfies $(\grad p)(x_0) \not= 0$ and $q_i(x_0)=0$,
  as well as $(\grad q_i)(x_0)\not= 0$, for $i=1,\dots, k$.
\begin{itemize}
\item[(i)] The function $p$ attains an isolated local maximum on $G$ at
$x_0$, if
\begin{equation*} 
(\grad p)(x_0) \in \interior N(x_0),
\end{equation*}
where $\interior N(x_0)$ is the \emph{interior} of the normal cone.
\item[(ii)] The function $p$ does not attain a local maximum on $G$ at $x_0$, if
\begin{equation*} 
(\grad p)(x_0) \not\in N(x_0).
\end{equation*}
\end{itemize}
\end{proposition}

\subsection{Proof of Theorem~\ref{th:qf_principal}}

First we compute the gradient of the Hermite invariant:

\begin{lemma}
\label{lem:gradient}
The Taylor series of the Hermite invariant $H$ at the quadratic function $f_0$ lying in the positive Erdahl cone is
\begin{equation*}
  \frac{1}{(\det Q_{f_0})^{1/n}}\left(\mu(f_0) 
  - \left(\ev_{c_{f_0}} + \frac{\mu(f_0)}{n} Q_{f_0}^{-1}, f-f_0\right) + \text{h.o.t.}\right),
\end{equation*}
where $\text{h.o.t.}$ stands for higher order terms.
\end{lemma}

\begin{proof}
The Taylor series of the functional $\mu$ at $f_0$ is
\begin{equation*}
\mu(f_0) - (\ev_{c_{f_0}},f - f_0) + \text{h.o.t.},
\end{equation*}
and the gradient of the determinant is $(\grad \det)(Q) = (\det Q) Q^{-1}$.
\end{proof}

We need the following convexity result. It implies that local maxima of
the Hermite invariant can only be attained at the extreme rays of the
positive Erdahl cone. This and the existence of these local maxima, which we will establish in the next section, shows that the
interior of the Erdahl cone is \emph{not} equal to the positive Erdahl cone; although it is of course contained in it.

\begin{lemma}
\label{lem:boundary}
  Let $f_1$ and $f_2$ be two quadratic functions in the positive Erdahl cone having positive Hermite invariants. Then, the maximum of the Hermite
  invariant $H$ on $\cone\{f_1, f_2\}$ is only attained at its extreme
  rays $\cone\{f_1\}$ or $\cone\{f_2\}$.
\end{lemma}

\begin{proof}
  We may assume that $f_1$ and $f_2$ are not collinear. Since $H$ is scale-invariant for positive scalars we may assume that $\mu(f_1) = \mu(f_2)$. It is sufficient to prove that
\begin{equation}
\label{eq:H_convex}
H(t f_1 + (1-t)f_2) < t H(f_1) + (1-t) H(f_2)
\end{equation}
holds for all $0 < t < 1$. The convexity of the function $\mu$ and the
convexity of the function $Q \mapsto (\det Q)^{-1/n}$, immediately
give the inequality~\eqref{eq:H_convex}, but only with ``$\leq$''
instead of ``$<$''.

Since the function $Q \mapsto (\det Q)^{-1/n}$ is
strictly convex (originally due to Minkowski \cite[\S 8]{Min}) we have equality in \eqref{eq:H_convex} if and only if
both functions
\begin{equation*}
  t \mapsto \mu(t f_1 + (1-t)f_2), \text{ and} \quad t \mapsto Q_{t f_1 +
    (1-t)f_2}
\end{equation*}
are constant for $0 \leq t \leq 1$. Suppose this is the case, then
\begin{equation*}
0 = \mu(t f_1 + (1-t)f_2) - t \mu(f_1) - (1-t) \mu(f_2) = -t(1-t)Q_{f_1}[b_{f_1} - b_{f_2}] ,
\end{equation*}
and hence $b_{f_1} = b_{f_2}$. From this it follows that $\alpha_{f_1} = \alpha_{f_2}$, and hence $f_1$, $t f_1 + (1-t)f_2$ and
$f_2$ all coincide which contradicts the assumption.
\end{proof}

Note that the lemma and its proof show that the function $H$ is strictly convex on the line segment connecting $f_1$ and $f_2$ if $\mu(f_1) = \mu(f_2)$ and if $\mu$ is positive on the line segment.

Now we can finish the proof.

\begin{proof}[Proof of Theorem~\ref{th:qf_principal}]
  Suppose that $f_0$ is perfect and eutactic.  Since the Hermite
  function is invariant with respect to positive scaling, we can work
  with the Erdahl cone intersected with the affine
  hyperplane~$H_{f_0}$ orthogonal to~$f_0$ and containing~$f_0$.
  Consider the set
\begin{equation*}
  G_{f_0} = \{f \in \EC_{>0} \cap H_{f_0} : (\ev_v, f) \geq 0, v \in \vertex \Del f_0\}.
\end{equation*}
Since $f_0$ is perfect, the functions $\ev_v$, with $v \in \vertex
\Del f_0$, span a subspace of codimension~$1$ in the
$\binom{n+2}{2}$-dimensional space of quadratic functions. Hence, for
a sufficiently small neighborhood $N_{f_0}$ of the point $f_0$ we have
\begin{equation*}
N_{f_0} \cap G_{f_0} = N_{f_0} \cap (\EC_{>0} \cap H_{f_0}).
\end{equation*}
Since $f_0$ is eutactic and because of the gradient computation in
Lemma~\ref{lem:gradient} we have that $-(\grad H)(f_0)$ lies in the interior of the inhomogeneous Voronoi cone $\cv(f_0)$. Here we take the interior within the affine hyperplane~$H_{f_0}$. Applying Proposition~\ref{prop:nlp}~(i) shows that $f_0$ is a local maximum of~$H$.

\medskip

Conversely, suppose that $f_0$ is extreme. Then by
Lemma~\ref{lem:boundary} we know that $f_0$ has to lie on an extreme
ray of the Erdahl cone, hence it is perfect.  Suppose that $f_0$ is
not eutactic. Proposition~\ref{prop:nlp}~(ii) shows that the only
situation which can occur is that $-(\grad H)(f_0)$ lies on the boundary of the inhomogeneous Voronoi cone $\cv(f_0)$.
Then, by Farkas' lemma (see e.g.\ Schrijver~\cite[Chapter 7.3]{Schrijver}), there exists a quadratic function $h$ in the affine
hyperplane~$H_{f_0}$ orthogonal to~$f_0$ and containing~$f_0$ so that
\begin{equation*}
\left\lbrace
\begin{array}{l}
(\ev_v, h) \geq 0, \quad \text{for all $v \in \vertex \Del f_0$,}\\
((\grad H)(f_0), h) = 0.
\end{array}
\right.
\end{equation*}
For $\lambda \geq 0$, consider the univariate function
\begin{equation*}
\varphi_{\alpha}(\lambda) = \mu(f_0 + \lambda(h + \alpha f_0)).
\end{equation*}
We can choose $\alpha$ so that
\begin{equation*}
0 = \frac{\partial\varphi_{\alpha}}{\partial \lambda}(0) = 
((\grad \mu)(f_0), h + \alpha f_0),
\end{equation*}
because $((\grad \mu)(f_0), f_0) = \mu(f_0) \neq 0$. Since
$\varphi_{\alpha}$ is convex and because
$\frac{\partial\varphi_{\alpha}}{\partial \lambda}(0) = 0$, we have
\begin{equation*}
\varphi_{\alpha}(\lambda) \geq \varphi_{\alpha}(0).
\end{equation*}
For $\lambda \geq 0$ consider the univariate function
\begin{equation*}
  \psi_{\alpha}(\lambda) = \det(Q_{f_0} + \lambda(Q_{h} + \alpha Q_{f_0}))^{-1/n}.
\end{equation*}
Since $\psi_{\alpha}$ is strictly convex, we have for $\lambda > 0$
\begin{equation*}
  \psi_{\alpha}(\lambda) > \psi_{\alpha}(0) + \frac{\partial \psi_{\alpha}}{\partial \lambda}(0)\lambda.
\end{equation*}
Taking the product shows
\begin{equation*}
H(f_0 + \lambda(h + \alpha f_0)) = \varphi_{\alpha}(\lambda) \psi_{\alpha}(\lambda) >
\varphi_{\alpha}(0) \psi_{\alpha}(0) = H(f_0),
\end{equation*}
because $\frac{\partial\psi_{\alpha}}{\partial \lambda}(0) \geq
0$. Hence, $f_0$ is not extreme.
\end{proof}

\section{Examples --- Strongly inhomogeneous perfect forms}
\label{sec:examples}

Venkov introduced strongly perfect forms in \cite
{Venkov}. \emph{Strongly perfect forms} are PQFs in which the shortest
vectors carry a spherical $4$-design.

\begin{theorem}[Venkov \cite{Venkov}]
\label{th:venkov}
Strongly perfect forms are extreme.
\end{theorem}

The notion of spherical designs is due to Delsarte, Goethals,
Seidel~\cite{DGS}.  Generally, finitely many points $X$ in~$\R^n$
carry a \emph{spherical $t$-design} (with respect to a PQF $Q$) if
they lie on a sphere
\begin{equation*}
  S_Q(c,r) = \{x \in \R^n : Q[x-c] = r^2\}, \quad \text{with $c \in \R^n$, and $r \in \R$},
\end{equation*}
and so that for all polynomials $f$ up to degree $t$ we have
\begin{equation*}
  \frac{1}{|X|} \sum_{x \in X} f(x) = \int_{S_Q(c,r)} f(x) d\omega(x),
\end{equation*}
where $\omega$ is the normalized surface measure on $S_Q(c,r)$.  The
maximal $t$ for which $X$ carries a spherical $t$-design is called its
\emph{strength} which we denote by~$s(X)$. An equivalent, alternative characterization of
spherical $t$-designs is the following: The points $X$ carry a
spherical $t$-design (with respect to a PQF $Q$) if there exists $c
\in \R^n$ and $r \in \R$ so that the following equalities hold for all
$k \leq t$ and all $y \in \R^n$:
\begin{equation*}
\sum_{x \in X} \left\langle Q, (x-c)(y-c)^t \right\rangle^k =
\left\{
\begin{array}{l}
0, \quad \text{for all odd $k$,}\\
\frac{1\cdot3\cdots(k-1)}{n(n+2)\cdots(n+k-2)} |X| r^{k/2} Q[y-c]^{k/2},\\
\qquad \text{for all even $k$.}
\end{array}
\right.
\end{equation*}

For the proof of Theorem~\ref{th:venkov} Venkov used Voronoi's
characterization of extreme PQFs in Theorem~\ref{th:Voronoi}. He
shows that having a spherical $2$-design already implies eutaxy, and
having a spherical $4$-design implies perfectness.

Theorem~\ref{th:venkov} gives a uniform way for showing that many remarkable PQFs are extreme. It applies e.g.\ to the forms of the root lattices $\mathsf{D}_4$, $\mathsf{E}_6$, $\mathsf{E}_7$, $\mathsf{E}_8$, the Coxeter-Todd lattice $\mathsf{K}_{12}$, the Barnes-Wall lattices $\mathsf{BW}_{2^d}$, with $d \geq 3$, the laminated lattice $\Lambda_{23}$, the shorter Leech lattice $\mathsf{O}_{23}$, the Leech lattice $\Lambda_{24}$, the Thompson-Smith lattice $\Lambda_{248}$.  All but the last case are treated in Venkov \cite{Venkov}. The result that the Barnes-Wall lattices are strongly perfect is due to Nottebaum \cite{Nottebaum}. For the Thompson-Smith lattice see Lempken, Schr\"oder, Tiep \cite{LST}. In the last two cases it is interesting to note that one can show the strong perfectness of $\mathsf{BW}_{2^d}$ and $\Lambda_{248}$ \emph{without} having the list of all minimal vectors (in fact at the time of writing not even the inhomogeneous minimum is known) but using properties of the automorphism group of $\mathsf{BW}_{2^d}$ and $\Lambda_{248}$ only.

Now we adapt the concept of strong perfection to the inhomogeneous
case.

\begin{definition}
  Let $Q$ be a PQF. It is called \emph{strongly inhomogeneous
    perfect}, if for each $c \in \R^n$ attaining $\mu(Q)$, the closest
  vectors $\Min_c Q$ carry a spherical $4$-design.
\end{definition}

\begin{theorem}
\label{th:strongly}
Inhomogeneous strongly perfect forms are inhomogeneous extreme.
\end{theorem}

We also adapt the definitions to the setting of quadratic functions.

\begin{definition}
  Let $f$ be a quadratic function lying in the positive Erdahl
  cone. It is called \emph{strongly perfect}, if the vertices of its
  Delone polytope carry a spherical $4$-design.
\end{definition}

\begin{theorem}
\label{th:qf_strongly}
Strongly perfect quadratic functions are extreme.
\end{theorem}

Like previously, Theorem~\ref{th:strongly} immediately
follows from Theorem~\ref{th:qf_strongly}. The proof of the second
theorem uses our characterization of inhomogeneous extreme forms in
Theorem~\ref{th:Voronoi}. It shows, like in the homogeneous case, that
having spherical $2$-designs already implies eutaxy, and that having
spherical $4$-designs implies perfectness.

\begin{proof}[Proof of Theorem~\ref{th:qf_strongly}]
  Let $f$ be a strongly perfect quadratic function. The set $X =
  \vertex \Del f$ carries a spherical $4$-design with respect to the
  quadratic form $Q_f$.

  We shall show that $f$ is eutactic: If we unfold the equation in the
  definition of eutactic quadratic functions, we get
\begin{equation*}
\left\lbrace\begin{array}{rcl}
1 & = & \sum\limits_{x \in X} \alpha_x,\\
0 & = & \sum\limits_{x \in X} \alpha_x(x - c_f),\\
\frac{\mu(f)}{n}Q_f^{-1} & = & \sum\limits_{x \in X} \alpha_x (x - c_f)(x - c_f)^t.
\end{array}\right.
\end{equation*}
We set $\alpha_x = \frac{1}{|X|}$ with $x \in X$, so that the first
condition in Definition~\ref{def:perfect_eutactic} (iii) is
satisfied. Then, by looking at the alternative definition of spherical
$1$- and $2$-designs, we see that the other two conditions are
satisfied, see e.g.\ \cite[Lemma 5.1]{SV1}.

  We shall show that $f$ is perfect: Let $g$ be a quadratic function
  which satisfies the linear equations
\begin{equation*}
  (\ev_x, g) = g(x) = 0 \quad \text{for all $x \in X$}.
\end{equation*}
Since $X$ carries a spherical $4$-design, we have
\begin{equation*}
  0 = \frac{1}{|X|} \sum_{x \in X} g(x)^2 = \int_{S_{Q_f}(c_f, \sqrt{\mu(f)})} g(x)^2 d\omega(x).
\end{equation*}
So, $g$ vanishes on $S_{Q_f}(c_f, \sqrt{\mu(f)}) = \{x \in \R^n : f(x)
= 0\}$. Hence, it has to be a multiple of $f$. So the space spanned by
the functions $\ev_x$, with $x \in X$, has codimension~$1$ in the
$\binom{n+2}{2}$-dimensional space of quadratic functions. In other
words, $f$ is perfect.
\end{proof}

 Using Theorem~\ref{th:qf_strongly} one can show that the PQFs
  belonging to the lattices $\mathsf{E}_6, \mathsf{E}_7$,
  $\mathsf{BW}_{16}$, $\Lambda_{23}$, $\mathsf{O}_{23}$ are
  inhomogeneous strongly perfect and hence inhomogeneous
  extreme. Geometrically this says that these 
  lattices yield local covering maxima. These are all inhomogeneous strongly PQFs we know of. In Table~\ref{table:stronglyperfect} we give some details about these PQFs and the Delone polytopes:
The second column gives the number of orbits of Delone polytopes. In all these cases there is only one orbit corresponding to points $c$ where $\mu(Q)$ is attained. In the last column we give a reference where a description of the orbits can be found.

 The PQFs belonging to the lattices $\Z^n$, $\mathsf{D}_n$, $\mathsf{E}_6^*$, $\mathsf{E}_7^*$, $\mathsf{E}_8$,  $\mathsf{K}_{12}$ are not inhomogeneous perfect. However they are
  inhomogeneous eutactic. We will get a geometrical interpretation
  from Theorem~\ref{th:pessimum}: These lattices yield local \emph{covering
    pessima}, i.e.\ the set of perturbations in which the covering
  density decreases has measure zero. Section~\ref{sec:pessima} is concerned with covering pessima.
In Table~\ref{table:inhomogeneouseutactic} we give some details about these PQFs and the Delone polytopes.  

  A PQF belonging to the Leech lattice is neither inhomogeneous
  perfect nor inhomogeneous eutactic. In fact, geometrically, the
  Leech lattice gives a local minimum for the covering density, see
  \cite{SV1}.

\renewcommand{\thetable}{\arabic{section}.\arabic{table}}

\begin{table}[htb]
\begin{tabular}{c|c|c|c|l}
name & \# orbits & $|\Min_c(Q)|$ & $s(\Min_c(Q))$ & reference\\
\hline
$\mathsf{E}_6$ & $1$ & $27$ & $4$ & Conway, Sloane \cite{CS2}\\
$\mathsf{E}_7$ & $2$ & $56$ & $5$ & CS \cite{CS2}\\
$\mathsf{BW}_{16}$ & $4$ & $512$ & $5$ & Dutour Sikiri\'c,\\
                              &          &          &         & Sch\"urmann, Vallentin \cite{DSV2}\\
$\mathsf{O}_{23}$ & $5$ & $94208$ & $7$ & DSV \cite{DSV2}\\
$\Lambda_{23}$ & $709$ & $47104$ & $7$ & DSV \cite{DSV2}\\
\end{tabular}
\\[0.3cm]
\caption{Lattices belonging to inhomogeneous strongly perfect forms.}
\label{table:stronglyperfect}
\end{table}

\begin{table}[htb]
\begin{tabular}{c|c|c|c|l}
name & \# orbits & $|\Min_c(Q)|$ & $s(\Min_c(Q))$ & reference\\
\hline
$\Z^n$ & $1$ & $2^n$ & $3$ & Conway, Sloane \cite{CS2}\\
$\mathsf{D}_3$ & $2$ & $6$ & $3$ & CS \cite{CS2}\\
$\mathsf{D}_4$ & $1$ & $8$ & $3$ & CS \cite{CS2}\\
$\mathsf{D}_n$, $n \geq 5$ & $2$ & $2^{n-1}$ & $3$ & CS \cite{CS2}\\
$\mathsf{E}_{6}^*$ & $1$ & $9$ & $2$ & CS \cite{CS2}\\
$\mathsf{E}_{7}^*$ & $1$ & $16$ & $3$ & CS \cite{CS2}\\
$\mathsf{E}_8$ & $2$ & $16$ & $3$ & CS \cite{CS2}\\
$\mathsf{K}_{12}$ & $4$ & $81$ & $3$ & Dutour Sikiri\'c,\\
                              &          &          &         & Sch\"urmann, Vallentin \cite{DSV2}
\end{tabular}
\\[0.3cm]
\caption{Lattices belonging to inhomogeneous eutactic forms. }
\label{table:inhomogeneouseutactic}
\end{table}

We finish this section by posing several problems:

\begin{itemize}
\item[(i)] Are there strongly perfect functions which do not define
  inhomogeneous strongly perfect forms?
\item[(ii)] Is a PQF of the Barnes-Wall lattice $\mathsf{BW}_{2^d}$ for $d \geq 5$
  inhomogeneous strongly perfect?
\item[(iii)] Is a PQF of the Thompson-Smith lattice $\Lambda_{248}$
  inhomogeneous strongly perfect?
\item[(iv)] It would be interesting to classify strongly perfect
  quadratic functions in low dimensions. So far only a classification
  up to dimension~$6$ is known. It is described in the next
  section. In the homogeneous case, strongly perfect forms have been
  classified up to dimension $12$ by Nebe and Venkov \cite{NV}.
\end{itemize}

\section{Finiteness and classification}
\label{sec:finiteness}

In this section we show that there are only finitely many inequivalent
perfect quadratic functions, respectively eutactic quadratic
functions, in a given dimension. Here, equivalence is defined using
scaling and using the action of the affine general linear group
\begin{equation*}
  \AGL_n(\Z) = \{u : \R^n \to \R^n : u(x) = v + Ax, \;\; \text{with $v \in \Z^n$ and $A \in \GL_n(\Z)$}\}.
\end{equation*}
More precisely, we say that two quadratic functions $f$ and $g$ are
\emph{equivalent} if there exists a positive scalar $\lambda$ and $u
\in \AGL_n(\Z)$ so that $f(x) = \lambda g(u(x))$.

\begin{theorem}\label{th:finiteness}
  In any dimension there are only finitely many inequivalent perfect
  quadratic functions, respectively weakly eutactic quadratic functions.
\end{theorem}

\begin{proof}
  From the work of Voronoi \cite[\S98]{Vor2} (see also Deza, Laurent
  \cite[Chapter 13.3]{DL}) it follows that, up to $\AGL_n(\Z)$
  equivalence, there are only finitely many Delone polytopes of
  quadratic functions. This implies that there are only finitely many
  inequivalent perfect quadratic functions.

  Now we argue that every Delone polytope $D$ determines up to
  equivalence at most one eutactic quadratic function. For this we
  define the cone
\begin{equation}
\label{eq:deltacone}
  \Delta(D) = \{f \in \EC_{>0} : \Del f = D\}.
\end{equation}
Since the function $\mu$ is strictly positive on it, 
Lemma~\ref{lem:boundary} and its proof show that $H$ has at most one critical point, which is a minimum of $H$.

If $f$ is weakly eutactic, then for all $g \in
\Delta(D)$ we have 
\begin{equation*}
\begin{array}{rcl}
(-(\grad H)(f), g)
& = & \displaystyle\frac{1}{(\det Q_{f})^{1/n}}\left(\ev_{c_{f}} + \frac{\mu(f)}{n}Q_{f}^{-1}, g\right)\\
& = & \displaystyle\frac{1}{(\det Q_{f})^{1/n}}\left(\sum_{v \in \vertex \Del f} \alpha_v \ev_{v}, g\right)\\
& = & 0,
\end{array}
\end{equation*}
and hence $f$ is a critical point of $H$. 
\end{proof}

Perfect quadratic functions have been classified up to dimension~$6$;
the classifications in dimension~$7$ and~$8$ seem to be complete:

\begin{description}
\item[Dimension 2, \dots, 5] Erdahl \cite[Theorem 5.1]{Erd} showed that
  there are no perfect quadratic functions in dimension $n = 2,
  \ldots, 5$.

\item[Dimension 6] Dutour \cite{Dut} showed that up to equivalence
  there is exactly one perfect quadratic function in dimension $6$: It
  is defined by the Schl\"afli polytope $2_{21}$ in dimension $6$
  having $27$ vertices (see e.g.\ \cite[Chapter 11.8]{Cox}). It is
  strongly perfect since the vertices of $2_{21}$ carry a spherical
  $4$-design.

\item[Dimension 7] In dimension $7$ there are two perfect quadratic functions known. The list is given in Dutour,
  Erdahl, Rybnikov \cite[Section 7]{DER}: One is
  defined by the Gosset polytope $3_{21}$ in dimension $7$ having $56$
  vertices (see e.g.\ \cite[Chapter 11.8]{Cox}). It is strongly
  perfect since the vertices of $3_{21}$ carry a spherical
  $5$-design. The other one is defined by the $35$-tope constructed by
  Erdahl, Rybnikov \cite{ER2}. It is eutactic (although the strength of the design is $0$), but it is not strongly perfect.

\item[Dimension 8] In dimension $8$ there are $27$ perfect quadratic functions known. They are described in Dutour,  Erdahl, Rybnikov \cite[Section 8]{DER}.  $21$ of them are eutactic, among them there is no strongly
  perfect quadratic function.
\end{description}

It would be interesting to understand the asymptotics of the number of perfect quadratic functions and the number of eutactic quadratic functions. At the moment it is not even clear whether the number grows with every dimension. This appears to be extremely likely: In dimension $9$ we found more than $100,000$ perfect quadratic functions.

\section{Pessima and topological Morse functions}
\label{sec:pessima}

In this section we study inhomogeneous eutactic forms. First we consider inhomogeneous eutactic forms which are not inhomogeneous perfect. They can be almost local maxima for the inhomogeneous Hermite invariant. By this we mean the following: A PQF is called a \emph{pessimum}, if it is not a local maximum of the inhomogeneous Hermite invariant, but for which almost all local perturbations decrease it. Note that there does not exist an analogue of pessima for the homogeneous Hermite invariant: There is no PQF for which almost all local perturbations increase the Hermite invariant. However, it is known (\u{S}togrin \cite{Stogrin}) that when a PQF is eutactic then the Hermite invariant decreases in almost every direction.

\begin{theorem}
\label{th:pessimum}
Let $Q$ be an inhomogeneous eutactic PQF which is not inhomogeneous extreme. Suppose for all quadratic functions $f$ lying in the positive Erdahl cone with $Q = Q_f$ and $\mu(Q) = \mu(f)$, the Delone polyhedron $\Del f$ is not a simplex. Then $Q$ is a pessimum.
\end{theorem}

\begin{proof}
Let $Q'$ be a generic perturbation of $Q$ so that all Delone polytopes of $Q'$ are simplices. Let $\Delta$ be a Delone simplex contained in a Delone polytope $D = \Del f$ of $Q$. Let $f'$ be the quadratic function with $\Del f = \Delta$ and $Q_{f'} = Q'$. Then we have the expansion
\begin{equation*}
H(f') = H(f) - \sum_{v \in \vertex D} \alpha_v (f'-f)(v) + \text{h.o.t.},
\end{equation*}
because $f$ is eutactic.  Since $D$ is not a simplex, there is a $v \in \vertex D$ so that $(f'-f)(v) > 0$. This implies that the second summand of the expansion is negative.
\end{proof}

This situation occurs for instance for the PQFs belonging to lattice given in Table~\ref{table:inhomogeneouseutactic}.

As a second application we show that the inhomogeneous Hermite invariant is generally not a topological Morse function.  We recall the following definition from Morse~\cite{Mor}.

\begin{definition}
  Let $M$ be an $m$-dimensional topological manifold and let $f$ be a
  real valued continuous function on $M$.
\begin{itemize}
\item[(i)] A point $q \in M$ is called \emph{topologically ordinary}
  if there exist neighborhoods $U$ of $q$ and $V$ of $0\in \R^m$ and a
  homeomorphism $\phi:V \to U$ such that for all $x \in V$
\begin{equation*}
\phi(0) = q, \;\; f(\phi(x))=x_1+f(q).
\end{equation*}
 Otherwise, it is called  \emph{topologically critical}.
\item[(ii)] A topologically critical point is called
  \emph{topologically non-degenerate} of index $r$ if there exist $U$,
  $V$, $\phi$ as above such that for all $x \in V$
\begin{equation*}
\phi(0) = q, \;\;  f(\phi(x))=-x_1^2-  \dots- x_r^2+x_{r+1}^2+\dots+x_m^2+f(q).
\end{equation*}
\item[(iii)] A function is called \emph{topological Morse function} if
  all points are either ordinary or topologically non-degenerate.
\end{itemize}
\end{definition}

Note that at a topological non-degenerate point the directions of decrease are homotopically equivalent to the sphere $S^{r-1} = \{x \in \R^{r} : \|x\| = 1\}$. The directions of increase are homotopically equivalent to the sphere $S^{m-r-1}$.

Since $H$ is scale invariant, it is not a topological Morse function for trivial reasons; the same is true for the homogeneous Hermite invariant $\gamma$. Ash \cite{Ash} showed that $\gamma$ is a topological Morse function on the cone of positive semidefinite $n \times n$-matrices where we mod out by positive scaling: $\mathcal{S}_{>0}^n / \R_{>0}$. As the following theorem shows, this is in general not the case for $H$.

\begin{theorem}\label{th:morse}
The inhomogeneous Hermite invariant is a topological Morse function on $\mathcal{S}_{>0}^n / \R_{>0}$ if and only if $n$ is at most three.
\end{theorem}

We need the following lemma:

\begin{lemma}
Let $Q$ be an inhomogeneous eutactic form. Then $Q$ is a topologically critical point for $H$ in $\mathcal{S}_{>0}^n / \R_{>0}$. It is a topologically non-degenerate point if and only if there exist one Delone polytope $D$ attaining the maximum circumradius such that for all Delone polytopes $D'$ attaining the maximum circumradius we have
\begin{equation*}
\lin\Delta(D') \subseteq\lin\Delta(D),
\end{equation*}
where $\Delta$ was defined in \eqref{eq:deltacone}.
\end{lemma}

\begin{proof}
Let $D_1, \ldots, D_r$ be the translation classes of Delone polytopes attaining the maximum circumradius. The argument in the proof of Theorem~\ref{th:pessimum} shows that $H$ increases in the direction of 
\begin{equation*}
U = \bigcup_{i=1}^r \lin\Delta(D_i) / \R_{>0}.
\end{equation*}
It decreases in all other directions. So it is a topologically critical point. If $U = \lin\Delta(D_i)$ for some $D_i$, then $Q$ is a topologically non-degenerate point. If $U$ is a union of subspaces which is not contained in $\lin\Delta(D_i)$ for one $D_i$, then $U$ is not homotopically equivalent to a sphere, so $Q$ is not a topologically non-degenerate point.
\end{proof}

\begin{proof}[Proof of Theorem~\ref{th:morse}]
There is at most one critical point in the secondary cone of a fixed Delone decomposition up to the action of $\GL_n(\Z)$. 

If $n$ equals two, there are two critical points: The PQF corresponding to the lattice $\Z^2$ and the one corresponding to the lattice $\mathsf{A}_2$. They are both inhomogeneous eutactic. In both cases there is only one Delone polytope up to translations and antipodality. So both PQFs are topologically non-degenerate by the previous lemma.

If $n$ equals three, there are five types of Delone subdivisions (due to the Russian crystallographer E.S.~Fedorov, see also Vallentin \cite{Vallentin}).  In all but the generic case one can check the following facts by inspection and elementary hand calculation: For every Delone subdivision which is not a triangulation there is a inhomogeneous eutactic PQF in which the Delone polytopes attaining the maximum circumradius are equivalent up to translations and antipodality. So we can apply the previous lemma, showing that these four points are topologically non-degenerate. In the generic case, where the subdivision is a triangulation, there is a PQF (associated to the lattice $\mathsf{A}^*_3$) where $H$ attains a local minimum.

If $n$ equals four, we consider the PQF which corresponds to the root lattice $\mathsf{D}_4$. It is inhomogeneous eutactic. There are three translation classes of Delone polytopes $D_1, D_2, D_3$ which are all regular cross polytopes realizing the circumradius.  Their linear subspaces $\lin\Delta(D_i)$ are not contained in each other, so by the preceding lemma the PQF is not topologically non-degenerate. 

For $n$ greater than four, we take the PQF which corresponds to the lattice $\mathsf{D}_4 \times \Z^{n-4}$.
\end{proof}

\section{An infinite series of inhomogeneous extreme forms}
\label{sec:series}

In this section we construct a series of inhomogeneous extreme forms for dimensions $n \geq 6$. The first two PQFs in the series correspond to the lattices $\mathsf{E}_6$ and $\mathsf{E}_7$. These PQFs were originally introduced in \cite{Dut2}.

For giving the construction and for its analysis it is convenient not to work with the standard lattice but with the lattice~$L_n$ which is spanned by the root lattice  $(\mathsf{D}_{n-1}, 0)$ and the vector $(-1/2,(1/2)^{n-2}, 1)$. It comes with the PQF
\begin{equation*}
Q_n[x] = 
\left\{
\begin{array}{ll}
x_1^2+\dots + x_{n-1}^2+(n-3)/4 x_n^2 & \text{if $n$ even,}\\
x_1^2+\dots + x_{n-1}^2+(n-5)/4 x_n^2 & \text{if $n$ odd}
\end{array}
\right.
\end{equation*}
We denote this pair by $[L_n, Q_n]$. We have $|\Aut([L_n, Q_n])| = |\Aut(\mathsf{D}_{n-1})|$.

\begin{theorem}
\label{th:series}
For $n\geq 6$, the lattice $[L_n, Q_n]$ are local covering maxima.
\end{theorem}

The main step of the computation is to prove that the big Delone polytope $P_n$ defined in the next section is the only one attaining the maximum circumradius. In order to show this we enumerate all Delone polytopes up to symmetry. We shall prove that our list is complete by a volume argument.

In the remaining part of this section will be used to give a proof of the theorem which is largely computational. The idea of the proof is based on the algorithms given in \cite{DSV2} which are implemented in \cite{polyhedral}.

In the proof we heavily rely on the computation of volumes of polyhedra: Let $P$ be a non-necessarily
full dimensional polytope of $\R^n$. By $\vol(P)$ we denote the volume of $P$ for the volume form induced by the scalar product on the affine space $\aff(P)$ defined by $P$. If $v\notin \aff(P)$, we will then have the relation
\begin{equation}
\label{ConeRelation}
\vol(\conv(P, v))=\frac{1}{\tdim(\conv(P,v))}\dist(v, \aff(P)) \vol(P),
\end{equation}
where $\conv(P,v)$ denotes the convex hull of the polytope $P$ and the point $v$, and where $\dist(v, \aff(P))$ denotes the Euclidean distance between $v$ and $\aff(P)$. An easy consequence of this formula is that if $\aff(P)$ is a hyperplane
of dimension $n-1$ defined by an affine equality $\phi(x)=0$, then we have for $v,v'\notin \aff(P)$ the relation
\begin{equation}
\label{AffEquaRelation}
\vol(\conv(P, v)) = \frac{|\phi(v)|}{|\phi(v')|}  \vol(\conv(P, v')).
\end{equation}
Relation (\ref{ConeRelation}) admits a generalization: If $P$, $Q$ are a $p$-, $q$-dimensional polytopes, then the
$1+p+q$-dimensional polytope $P \times Q$ defined as
\begin{equation*}
P\times Q=\conv((0, P, 0^{q}), (1, 0^{p}, Q))
\end{equation*}
has volume
\begin{equation}
\label{VolumeProductFormula}
\vol(P\times Q) = \vol(P)\vol(Q)\frac{p! q!}{(1+p+q)!}.
\end{equation}
In the following we use the notation 
\begin{equation*}
\frac{1}{2}H_n = \left\{ x \in \{0,1\}^n : \sum_{i=1}^n x_i \text{ even}\right\}.
\end{equation*}
for the half cube.

\subsection{The big Delone polytope}

As we shall prove later, there is only one Delone polyope of $[L_n,Q_n]$ where the maximum circumradius is attained. It is the polytope $P_n$ which is defined as follows. If $n$ is even then $P_n$ has the vertices
\begin{equation*}
((1/2)^{n-1}, 1) \pm e_i, i = 1, \ldots, n-1,\;\; ((1/2)^{n-1}, -1), \;\; (\frac{1}{2}H_{n-1},0),
\end{equation*}
If $n$ is odd, then $P_n$ has the vertices
\begin{equation*}
((1/2)^{n-1}, \pm1) \pm e_i, i = 1, \ldots, n-1,\;\; (\frac{1}{2}H_{n-1},0).
\end{equation*}
The squared circumradius of $P_n$ is
\begin{equation*}
\mu_{P_n}=
\left\{
\begin{array}{ll}
(n-2)^2/(4(n-3)), & \text{if $n$ even,}\\
(n-1)/4,         & \text{if $n$ odd.}\\
\end{array}
\right.
\end{equation*}
The center of $P_n$ is
\begin{equation*}
c_{P_n}
=
\left\{
\begin{array}{ll}
((1/2)^{n-1}, 1/(n-3)), & \text{if $n$ even,}\\
((1/2)^{n-1}, 0),           & \text{if $n$ odd.}
\end{array}
\right.
\end{equation*}

It is proved in \cite{Dut2} that $P_n$ uniquely determines $[L_n, Q_n]$ if $n \geq 6$. So the quadratic function $f_n$ corresponding to $P_n$ is inhomogeneous perfect. It is also inhomogenous extreme:

\begin{lemma}
The quadratic function $f_n$ is inhomogeneous eutactic.
\end{lemma}
\begin{proof}
 The polytope $P_n$ has three orbits of vertices if
  $n$ is even which can be distinguished by considering the last coordinate: $-1$, $0$, $+1$.
Then, the following coefficients satisfy  the eutaxy condition
\begin{eqnarray*}
a_{-1} & = & (n-2)/(2n(n-3)^2),\\
a_0 & = & ((n-2)(n^2-5n+2))/(2^{n-2} n (n-3)^2),\\
a_1 & = & 2/(n(n-3)^2).
\end{eqnarray*}
The polytope $P_n$ has only two orbits of vertices if $n$ is odd which can be distinguished by considering the last coordinate: $\pm 1$, $0$. Then, the following coefficients satisfy the eutaxy condition
\begin{eqnarray*}
a_{\pm 1} & = &  1/(4n(n-5)),\\
a_0 & = & (n^2-6n+1)/(2^{n-2}n(n-5)).
\end{eqnarray*}
\end{proof}

The lower bound on the volume of $P_n$ will turn out to be tight.

\begin{lemma}
The volume of $P_n$ is at least $V_n$ where
\begin{equation*}
\begin{split}
& V_n = 2(n-1)\frac{1}{n(n-1)}\left( 1-\frac{2^{n-3}}{(n-2)!}\right) +2^{n-2} 2^{n-3} \frac{n-3}{n!}\\
& \qquad  +\sum_{j=3}^{n-3} \frac{2^{n-2}(n-1)!}{(i+1)! 2^{j-1}j!} (j!-2^{j-1})\frac{n-j-1}{2n!}\\
& \qquad +\frac{2^{n-1}}{n!}+2^{n-2}\frac{n-1}{2n!}+2^{n-2} \frac{n-3}{2n!},
\end{split}
\end{equation*}
if $n$ is even, and 
\begin{equation*}
\begin{split}
& V_n = 2(n-1)(n-2)\frac{4}{n(n-1)(n-2)}\left( 1-\frac{2^{n-4}}{(n-3)!}\right)+2^{n-1} \frac{n-1}{2n!}\\
& \qquad +\sum_{j=3}^{n-4} \frac{2^{n-1}(n-1)!}{(i+1)! 2^{j-1} j!}(j!-2^{j-1})\frac{n-j-1}{2 n!}\\
& \qquad + 2^{n-1} 2^{n-3}\frac{n-3}{n!}+2 \frac{2^{n-1}}{n!}+2^{n-2}(n-1) \frac{n-4}{n!},
\end{split}
\end{equation*}
if $n$ is odd.
\end{lemma}

\begin{proof}
Denote by ${\mathcal F}(P)$ the set of facets of $P$ and by $c$ the point $((1/2)^{n-1}, 0)$. We have 
\begin{equation*}
\vol(P_n)=\sum_{F\in {\mathcal F}(P_n)} \vol(\conv(F, c)).
\end{equation*}
Since $c$ is invariant under the automorphism group of $P_n$, the above
sum can be grouped by orbits of facets of $P_n$.

Below, we list the facets $F$ of $P_n$. The first line gives the separating hyperplane, the second line contains a list of incident vertices, the third line contains the volume $\vol(\conv(F, c))$ and the last line contains the size of the orbits. We frequently make use of the transformation $g$ defined by
\begin{equation*}
g(x_1, x_2, \dots, x_n) = (1-x_1, x_2, \dots, x_n).
\end{equation*}

\begin{itemize}
\item Facet $F_1$: a cross polytope
\begin{itemize}
\item $\sum_{j=1}^{n-1} x_j + (n-5)/2 x_n\geq 1$,
\item $g(e_j)$, $g(((1/2)^{n-1}, 1) - e_j)$ for $1\leq j\leq n-1$,
\item $2^{n-3}/n! (n-3)$,
\item $2^{n-2}$.
\end{itemize}
\item Facet $F_2$: a cross polytope
\begin{itemize}
\item $x_n\leq 1$,
\item $((1/2)^{n-1}, 1)\pm e_j$ for $1\leq j\leq n-1$,
\item $2^{n-1}/n!$,
\item $1$ if $n$ even, $2$ if $n$ odd.
\end{itemize}
\item Facet $F_3$: simplex
\begin{itemize}
\item $\sum_{i=1}^{n-1} x_i +(n-3)/2 x_n\geq 0$,
\item $0$, $((1/2)^{n-1}, 1) - e_j$ for $1\leq j\leq n-1$,
\item $(n-1)/2n!$,
\item $2^{n-2}$ if $n$ even, $2^{n-1}$ if $n$ odd.
\end{itemize}
\item Facet $F_4$: only if $n$ even
\begin{itemize}
\item $2x_{1} - x_{n}\geq 0$,
\item $((1/2)^{n-1}, 1)$, $(0, \frac{1}{2}H_{n-2}, 0)$ and $(-1/2, (1/2)^{n-2}, -1)$,
\item $1/n(n-1) \left( 1-2^{n-3}/(n-2)!\right)$.
\item $2(n-1)$.
\end{itemize}
\item Facet $F_5$: simplex, only if $n$ even
\begin{itemize}
\item $\sum_{i=1}^{n-1} x_i +(n-1)/2 x_n\geq 1$, 
\item $((1/2)^{n-1}, -1)$, $g(e_j)$ for $1\leq j\leq n-1$,
\item $(n-3)/2n!$,
\item $2^{n-2}$.
\end{itemize}
\item Facet $F_6$: only if $n$ odd
\begin{itemize}
\item $x_1+x_2\geq 0$,
\item $((1/2)^{n-1}, \pm 1) - e_j$ for $j=1$, $2$, $(0, 0, \frac{1}{2}H_{n-3}, 0)$,
\item $4/(n(n-1)(n-2)) \left(1 - 2^{n-4}/(n-3)!\right)$,
\item $2(n-1)(n-2)$.
\end{itemize}
\item Facet $F_7$: simplex, only if $n$ odd
\begin{itemize}
\item $\sum_{i=1}^{n-2}x_i+(n-4) x_{n-1} \geq 1$, 
\item $g(e_j)$ for $1\leq j\leq n-2$, $((1/2)^{n-2}, -1/2, \pm 1)$,
\item $(n-4)/n!$,
\item $2^{n-2}(n-1)$.
\end{itemize}
\item Facet $F_{i,j}$: for $i+j=n-2$, $j\geq 3$ and $i\geq 1$ for $n$ even, $i\geq 2$ for $n$ odd
\begin{itemize}
\item $\sum_{k=j+1}^{n-1} x_j +(1-i)/2 x_n\geq 0$,
\item $(\frac{1}{2} H_j, 0^{i+1},0)$, $((1/2)^{n-1},1)-e_k$ for $j+1\leq k\leq n-1$,
\item $\left( j!-2^{j-1}\right) (n-j-1)/2(n!)$,
\item $(i+1)! 2^{j-1} j!$.
\end{itemize}
\end{itemize}
\end{proof}

\subsection{Proof of Theorem~\ref{th:series}}

We only have to show that for every Delone polytope $P$ of $[L_n,Q_n]$ which is not equivalent to $P_n$ we have $\mu_P < \mu_{P_n}$. 

We now construct the remaining classes of Delone polytopes of $[L_n, Q_n]$: If $n$ is even we have one additional class and if $n$ is odd we two additional classes.

\begin{itemize}
\item 
If $i+j=n-1$ and $3\leq i\leq j$, we denote by $H_{i,j}$ the polytope with vertices
\begin{equation*}
(\frac{1}{2}H_i, 0^{j-1}, 0),  ((1/2)^{i}, (1/2)^j-g(\frac{1}{2}H_j), 1).
\end{equation*}
The size of the stabilizer is
\begin{equation*}
|\Stab(H_{i,j})|=\left\lbrace\begin{array}{ll}
2^{i-1} i! 2^{j-1}j!, &  \textrm{if }i\not= j,\\
2\times 2^{i-1} i! 2^{j-1}j!, &  \textrm{if }i= j.
\end{array}\right.
\end{equation*}
Using the formula for the product polytope we get
\begin{equation*}
\vol(H_{i,j})=\left(1-\frac{2^{i-1}}{i!}\right)\left(1-\frac{2^{j-1}}{j!}\right)\frac{i! j!}{n!}
=(i!-2^{i-1})(j!-2^{j-1})\frac{1}{n!}.
\end{equation*}
We set $C=n-3$ if $n$ is even and $C=n-5$ if $n$ is odd.
The center of $H_{i,j}$ is
\begin{equation*}
c_{H_{i,j}}=((1/2)^i, 0^j, \alpha), \text{ with }\alpha=\frac{C+j-i}{2C}.
\end{equation*}
The squared radius of the sphere around $H_{i,j}$ is
\begin{equation*}
\mu_{H_{i,j}}=\frac{C^2+2C(n-1)+(j-i)^2}{16C} < \mu_{P_n}.
\end{equation*}

\item
If $n$ is odd, then the simplex $S_n$ with vertex set
\begin{equation*}
0, (0^{n-1}, 2),  ((1/2)^{n-1}, 1) - e_j, \text{ with } j = 1, \ldots, n-1,
\end{equation*}
is a Delone polytope. We have 
\begin{eqnarray*}
|\Stab(S_n)| & = & 2(n-1)!,\\
\vol(S_n) & = & \frac{n-3}{n!},\\
c_{S_n} & = & ((1/(n-3))^{n-1}, 1),\\
\mu_{S_n} & = & \frac{n-5}{4}+\frac{n-1}{(n-3)^2} < \mu_{P_n}.
\end{eqnarray*}
\end{itemize}

Now we finish the proof by a volume computation showing that our list of orbits is complete. Denote by $O(D_1)$, \dots, $O(D_r)$ the orbits of Delone polytope of $[L_n,Q_n]$ of representative $D_i$.
On the one hand, we have 
\begin{equation*}
2=\sum_{i=1}^{r} |O(D_i)| \vol(D_i).
\end{equation*}
One the other hand, we have the equality
\begin{equation*}
2= \sum_{i=1}^{\frac{n-2}{2}} |O(H_{i,j})| \vol(H_{i,j}) + 2 V_n,
\end{equation*}
if $n$ is even, and
\begin{equation*}
2 = |O(S_n)|\vol(S_n)+\sum_{i=1}^{\frac{n-1}{2}} |O(H_{i,j})| \vol(H_{i,j}) + V_n,
\end{equation*}
if $n$ is odd. This implies that $\vol(P_n)=V_n$ and that the list of orbits of Delone
polytopes is complete. This finishes the proof of the theorem.

\section*{Acknowledgements}

We thank Peter McMullen for proposing the name \emph{covering
  pessima}. The third author thanks Rudolf Scharlau for his suggestion to work on the covering problem and Joseph Oesterl\'e for an interesting discussion during the DIAMANT symposium in November 2010. We thank the referee for careful reading of our manuscript and insightful comments.

We started this research during the Junior Trimester Program (February
2008--April 2008) on ``Computational Mathematics''. Then, part of
this research was done at the Mathematisches Forschungsinstitut
Oberwolfach during a stay within the Research in Pairs Programme from
May 3, 2009 to May 16, 2009. We thank both institutes for their
hospitality and support.

\end{document}